\documentclass[12pt,oneside]{amsart}
\usepackage[left=1.35in,right=1.35in,top=1in,bottom=1in]{geometry}
\usepackage{amsmath,amsthm,amssymb,comment,enumitem}
\usepackage{leftidx}
\usepackage{graphicx}
\newcommand{\mybinom}[3][0.9]{\scalebox{#1}{$\dbinom{#2}{#3}$}}
\usepackage[pdftex,bookmarks=true,breaklinks]{hyperref}
\hypersetup{
  colorlinks=true,
  linkcolor=blue,
  filecolor=blue,
  urlcolor=blue,
  citecolor=blue
}
\newtheorem{theorem}{Theorem}[section]
\newtheorem{proposition}[theorem]{Proposition}
\newtheorem{corollary}[theorem]{Corollary}
\newtheorem{lemma}[theorem]{Lemma}

\theoremstyle{definition}
\newtheorem{definition}[theorem]{Definition}

\newtheorem{problem}[theorem]{Problem}
\newtheorem{notation}[theorem]{Notation}

\theoremstyle{remark}
\newtheorem{remark}[theorem]{Remark}

\numberwithin{equation}{section}

\newcommand{\R}{\mathbb{R}}

\newcommand{\C}{\mathbb{C}}
\newcommand{\LP}{\mathcal{LP}}

\newcommand{\pd}{\partial}

\newcommand{\He}{H\!e}

\allowdisplaybreaks

\begin{document}
  
  \title[Interlacing of Coefficient Polynomials]{Interlacing Properties
    of Coefficient Polynomials in Differential Operator Representations of
    Real-Root Preserving Linear Transformations}
  
  % author one information
  \author[D.~A.~Cardon]{David~A.~Cardon}
  \address{Department of Mathematics, Brigham Young University, Provo, UT 84602, USA}
  \email{cardon@mathematics.byu.edu}
  
  % author two information
  \author[E.~L.~Sorensen]{Evan~L.~Sorensen}
  \address{Department of Mathematics, University of Wisconsin, Madison, WI 53706, USA}
  \email{elsorensen@wisc.edu}

  % author three information
  \author[J.~C.~White]{Jason~C.~White}
  %\address{Department of Mathematics, Brigham Young University, Provo, UT 84602, USA}
  \email{white.jason.c@gmail.com}

  \date{December 2021}
  
  \begin{abstract}
    We study linear transformations $T \colon \R[x] \to \R[x]$ of the form $T[x^n]=P_n(x)$ where $\{P_n(x)\}$ is a real orthogonal polynomial system. With $T=\sum \tfrac{Q_k(x)}{k!}D^k$, we seek to understand the behavior of the transformation $T$ by studying the roots of the $Q_k(x)$. We prove four main things. First, we show that the only case where the $Q_k(x)$ are constant and $\{P_n(x)\}$ is an orthogonal system is when the $P_n(x)$ form a shifted set of generalized probabilist Hermite polynomials. Second, we show that the coefficient polynomials $Q_k(x)$ have real roots when the $P_n(x)$ are the physicist Hermite polynomials or the Laguerre polynomials. Next, we show that in these cases, the roots of successive polynomials strictly interlace, a property that has not yet
    been studied for coefficient polynomials. We conclude by discussing the Chebyshev and Legendre polynomials, proving a conjecture of Chasse, and presenting several open problems.
  \end{abstract}
  
  \maketitle

  %%%%%%%%%%%%%%%%%%
  %Introduction
  %%%%%%%%%%%%%%%%%
  \section{Introduction}  \label{sec:Introduction} Let $T\colon \R[x]\to \R[x]$ be a linear transformation such that for every real-rooted polynomial $p(x)$, the polynomial $T[p(x)]$ has only real roots. 
  The study of such transformations is motivated by the Riemann Hypothesis. In recent years, transformations involving orthogonal polynomials have been considered because they often naturally arise when studying more general linear operators (see for example \cite{Chasse-PhD-2011}, \cite{Fisk-Laguerre-Polynomials-2008}, \cite{Forgacs-Piotrowski-Hermite-2015}, \cite{Piotrowski-PhD-2007}). We are interested in transformations $T$ with the real-root preserving property and the additional condition that for all $n$, $T[x^n] = P_n(x)$, where the set $\{P_n(x)\}_{n=0}^{\infty}$ is an orthogonal polynomial system.
  
  Any linear operator $T \colon \C[x] \to \C[x]$ has a unique representation~\cite[Prop.~29, p.~32]{Piotrowski-PhD-2007} of the form \begin{equation} \label{eqn:differential-operator-representation} T=\sum_{k=0}^{\infty} \frac{Q_k(x)}{k!}D^k,
  \end{equation}
  where $D$ denotes differentiation and the $Q_k(x)$ are complex polynomials. Motivated by the work of Chasse~\cite{Chasse-PhD-2011} and Forg\'acs and Piotrowski~\cite{Forgacs-Piotrowski-Hermite-2015}, we study the zeros of the coefficient polynomials $Q_k(x)$.
  \begin{notation}
    %We discuss six basic linear operators $T \colon \R[x] \to \R[x]$.
    We use the following notation.
    \begin{enumerate}
      \item[(i)] $P_n(x)$ is an arbitrary $n$\textsuperscript{th} degree polynomial.
      \item[(ii)] $\He_n^\alpha(x)$ are the generalized probabilist Hermite polynomials with parameter $\alpha$.
      \item[(iii)] $H_n(x)$ are the physicist Hermite polynomials.
      \item[(iv)] $L_n(x)$ are the Laguerre polynomials $L^{\alpha}_n(x)$ for the parameter $\alpha=0$, where we follow the standard notation $L_n(x) = L^0_n(x)$ as explained in \cite[p.145]{Chihara-Orthogonal-Polynomials-1978}, \cite[p.28]{Rahman-Schmeisser-2002}, and \cite[\S8.970]{Gradshteyn-Ryzhik-2000}.
      \item[(v)] $P^T_n(x)$ are the Chebyshev polynomials.
      \item[(vi)] $P^{Le}_n(x)$ are the Legendre polynomials.
    \end{enumerate}
    %\begin{enumerate}
    % \item[(i)]
    %  $T[x^n]=P_n(x)$ where $P_n(x)$ is an arbitrary $n$\textsuperscript{th} degree polynomial,
    %\item[(ii)]
    % $T[x^n] = \He_n^\alpha(x)$ where $\He_n^\alpha(x)$ is the $n$\textsuperscript{th} generalized probabilist Hermite polynomial with parameter $\alpha$,
    %\item[(iii)]
    %  $T[x^n] = H_n(x)$ where $H_n(x)$ is the $n$\textsuperscript{th} physicist Hermite polynomial,
    %\item[(iv)]
    %  $T[x^n] = L_n(x)$ where $L_n(x)$ is the $n$\textsuperscript{th} Laguerre polynomial $L^{\alpha}_n(x)$ for the parameter $\alpha=0$, where we follow the standard notation $L_n(x) = L^0_n(x)$ as explained in \cite[p.145]{Chihara-Orthogonal-Polynomials-1978}, \cite[p.28]{Rahman-Schmeisser-2002}, and \cite[\S8.970]{Gradshteyn-Ryzhik-2000},
    %\item[(v)]
    % $T[x^n] = P^T_n(x)$ where $P^T_n(x)$ is the $n$\textsuperscript{th} Chebyshev polynomial,
    %\item[(vi)]
    % $T[x^n] = P^{Le}_n(x)$ where $P^{Le}_n(x)$ is the $n$\textsuperscript{th} Legendre polynomial.
    %\end{enumerate}
    For each of these polynomial systems, we consider the transformation $T[x^n] = P_n(x)$. To distinguish the polynomials $Q_n(x)$ from equation~\eqref{eqn:differential-operator-representation} for these different cases, we will write $Q_n(x)$, $Q_n^{\He}(x)$, $Q_n^H(x)$, $Q_n^{L}(x)$, $Q_n^T(x)$, and $Q_n^{Le}(x)$, respectively.
  \end{notation}
  
  We will abbreviate orthogonal polynomial sequences by writing OPS in the singular and plural senses. We hope that studying representations of the form~\eqref{eqn:differential-operator-representation} for known transformations that preserve real-rootedness will give insight into general real-root preserving transformations.
  
  Recently, Chasse~\cite{Chasse-PhD-2011} and Forg\'acs and Piotrowski~\cite{Forgacs-Piotrowski-Hermite-2015} have studied properties of the polynomials $Q_k(x)$ when $T$ is a real-root preserving transformation. In this paper, we will show some examples where $T$ is a real-root preserving transformation, the $Q_k(x)$ are real-rooted, and the $Q_k(x)$ have interlacing roots. More specifically:
  \begin{enumerate}
    \item
    We classify all orthogonal polynomial systems $\{P_n(x)\}_{n=0}^{\infty}$ in the case when each $Q_k(x)$ is a constant (Section~\ref{section:constant_coeffs}).
    \item
    We find explicit formulas for $Q_k^H(x)$ and show that $Q_k^H(x)$ and $Q_{k+1}^H(x)$ have real interlacing roots (Section~\ref{section:Hermite}).
    \item
    We find explicit formulas for $Q_k^{L}(x)$ and show that $Q_{k}^{L}(x)$ and $Q_{k+1}^L(x)$ have real interlacing roots (Section~\ref{section:Laguerre}).
    \item
    We demonstrate a generating function for the polynomials $Q_k^{L}(x)$ (Section~\ref{section:generating function}).
    \item
    We prove a conjecture of Chasse~\cite{Chasse-PhD-2011} that states that the coefficient polynomials $Q_{k}^{Le}(x)$ satisfy $Q_{2k + 1}^{Le}(x) = 0$ and $Q_{2k}^{Le}(x) = C_{2k}(x^2 - 1)^k$ for constants $C_{2k}$, $k \ge 0$ (Section~\ref{section:ChebLeg}).
  \end{enumerate}
  
  These results give some insight into an open problem, which is to determine necessary and sufficient conditions on the $Q_k(x)$ to guarantee that the linear operator $T$ is a real-root preserving operator. The next section provides some background from the literature, and subsequent sections contain new results, as outlined above. Section~\ref{section:open_prob} presents several open problems.
  
  \section{Previous results}
  
  Chasse~\cite{Chasse-PhD-2011} showed how to obtain an explicit representation for the coefficient polynomials of an arbitrary linear operator $T \colon \C[x] \to \C[x]$, as follows:
  \begin{proposition}[\cite{Chasse-PhD-2011}, Proposition 216]
    \label{prop:Piotrowski-differential-operators}
    If $T \colon  \C[x] \to \C[x]$ is a linear operator, then in the  representation $T = \sum_{n=0}^\infty \frac{Q_n(x)}{n!} D^n$, the  coefficient polynomials are given by
    \begin{equation} \label{eqn:Q-general-expression}
      Q_n(x) =  \sum_{k=0}^n \mybinom{n}{k} T[x^k] (-x)^{n-k}.
    \end{equation}
  \end{proposition}
  
  We call a linear transformation $T$ \textit{real-root preserving} if, whenever a polynomial $f(x) \in \R[x]$ has only real roots, the polynomial $T[f(x)]$ is guaranteed to have only real roots. We say the operator $T$ is \textit{stability preserving}, if all roots of $T(f)$ lie in the upper (or lower) half plane whenever $f(x)$ has its roots in that same half plane. We say that $T$ is \textit{stability reversing} if $T(f)$ has its roots in the opposite half plane whenever $f$ has all of its roots in either the upper or lower half plane. %Many authors use the term \textit{hyperbolicity preserving} to denote this property. 
  %In the following theorem, $\mathcal{H}_n(\R)$ denotes the class of polynomials $f(z_1,\ldots, z_n) \in \R[z_1,\ldots,z_n]$ such that $f(z_1, \ldots ,z_n) \neq 0$ whenever $\operatorname{Im} (z_j)>0$ for $1 \leq j \leq n$. Furthermore, $\overline{\mathcal{H}}_n(\R)$ denotes the set of entire functions in $n$ variables that are uniform limits on compact sets of polynomials in $
  %\mathcal{H}_n(\R)$. 
  The following theorem classifies all nondegenerate real-root preserving linear transformations.

  \begin{proposition}[Borcea and Br\"and\'en \cite{Borcea-Branden-2009}, Cor.~1, see also Chasse~\cite{Chasse-PhD-2011}, pg 107] \label{Borcea-Branden-prop}
    A linear operator $T\colon \R[x] \to \R[x]$, with range of dimension greater than $2$, is real-root preserving if and only if $T$ is stability preserving or stability reversing. 
    %A linear operator $T\colon \R[x] \to \R[x]$ preserves hyperbolicity if and only if either
    %\begin{enumerate}
    % \item
    %      $T$ has range of dimension at most $2$ and is of the form
    %     \[
    %      T(f) = \alpha(f)P+\beta(f)Q \text {  for } f \in \R[z],
    %   \]
    %  where $\alpha,\beta\colon  \R[x] \to \R$ are linear functionals, and $P,Q \in \mathcal{H}_1(\R)$ have interlacing zeros, or
    
    %\item
    %$T[(z+w)^n] \in \mathcal{H}_2(\R) \cup \{0\}$ for all $n$, or
    %\item
    %     $T[(z-w)^n] \in \mathcal{H}_2(\R) \cup \{0\}$ for all $n$.
    %\end{enumerate}
  \end{proposition}
  
  %\begin{remark} \label{rmk:preserving,revesing}
  % Chasse~\cite{Chasse-PhD-2011} noted that condition (2) implies that the operator $T$ is \textit{stability preserving}, which means that all roots of $T(f)$ lie in the upper (or lower) half plane whenever $f(x)$ has its roots in that same half plane. Condition (3) implies that $T$ is \textit{stability reversing}, which means that $T(f)$ has its roots in the opposite half plane whenever $f$ has all of its roots in either the upper or lower half plane. It is important to note that a non-degenerate real-root preserving transformation is either stability preserving or stability reversing. Therefore, to classify such a transformation, it suffices to check the location of the roots of the image of a single polynomial with roots in the upper half plane.
  %\end{remark}

  From some examples of Forg\'acs and Piotrowski~\cite{Forgacs-Piotrowski-Hermite-2015} as well as Chasse~\cite{Chasse-PhD-2011}, we know that for
  \[
  T = \sum_{k=0}^\infty \frac{Q_k(x)}{k!} D^k,
  \]
  if the $Q_k(x)$ have  only real roots, it does not necessarily follow that $T$ is real-root preserving. Conversely, if $T$ is real-root preserving, it is not necessarily true that the $Q_k(x)$ have only real roots (see Example 212 in~\cite{Chasse-PhD-2011}). However, the following is known.
  
  \begin{proposition}[\cite{Chasse-PhD-2011}, Prop.~209] \label{Chasse-finite-operator}
    If the linear operator $T$ is a real-root preserver and if $T$ can be represented as a differential operator of finite order,
    \begin{equation}
      T = \sum_{k=0}^N \frac{Q_k(x)}{k!} D^k \qquad (Q_k(x) \in \R[x]),
    \end{equation}
    then the $Q_k(x)$ have only real zeros.
  \end{proposition}
  
  \begin{proposition}[\cite{Chasse-PhD-2011}, Prop.~217]
    \label{prop:Chasse-stability-reversing} Let $T\colon \C[x]\to\C[x]$ be a linear real-root preserver. If $T$ is stability reversing, then the coefficient polynomials $Q_k(x)$ are either hyperbolic \textup{(}real-rooted\textup{)} or identically zero.
  \end{proposition}
  
  The following class of functions is a major object of study in the field.
  \begin{definition}
    The \textit{Laguerre-P\'olya class}, denoted by $\LP$, is the set of functions obtained as uniform limits on compact sets of real polynomials with real roots. They have the Weierstrass product representation
    \begin{equation} \label{LP_Weirstrass}
      cz^n e^{-\alpha z^2 + \beta z}
      \prod_{k=1}^{\omega}\left(1-\frac{z}{a_k}\right)e^{\frac{z}{a_k}}\,,
    \end{equation}
    where $0 \leq \omega \leq \infty$ and $c$,  $\alpha$, $\beta$, $a_k$ are real, $n$ is a non-negative integer, $\alpha \geq 0$, and $\sum_{k=1}^{\omega}|a_k|^{-2} < \infty$.
  \end{definition}
  With this definition in place, the following well-known theorem, originally proved by P\'olya, gives a class of real-root preserving operators.
  
  \begin{proposition}[{\cite[Thm.~5.4.13, p.~157]{Rahman-Schmeisser-2002}}, \cite{Polya-1915}]
    \label{thm:LP-real-rootedness} Assume
    \[
    \phi(z) = \sum_{k=0}^{\infty} a_k z^k \in \LP.
    \]
    Then, if $f(z)$ is a real polynomial with only real-roots, $\phi(D)f(z)$ is also a real polynomial with only real roots.
  \end{proposition}
  In the previous proposition, sometimes we can guarantee that the zeros of $\phi(D)f(x)$ are simple.
  \begin{proposition}[Cardon and de Gaston \cite{Cardon-deGaston-2005}]
    In Proposition~\ref{thm:LP-real-rootedness}, if $\phi(z)$ has infinitely many zeros, then the zeros of the polynomial $\phi(D)f(z)$ are simple.
  \end{proposition}
  For a more detailed presentation of the Laguerre-P\'olya class, as well as the effect of various linear operators on the location of zeros, we highly recommend chapters VIII and XI of
  Levin~\cite{Levin}.
  
  \begin{definition} \label{def:Appell}
    An \textit{Appell} sequence is a sequence of polynomials satisfying $P_n'(x) = nP_{n-1}(x)$, for all $n \geq 0$. 
  \end{definition}
  
  \begin{definition} \label{def:interlace}
    A real-rooted system of polynomials $\{P_n(x)\}_{n = 0}^\infty$ with roots $x_{n,1} \le x_{n,2} < \cdots < x_{n,n}$ has \textit{interlacing roots}, if for all $n\ge 1$,
    \[
    x_{n + 1,1} \le x_{n,1} \le x_{n + 1,2} \le \cdots \le  x_{n,n} \le x_{n + 1,n + 1}.
    \]
    We say the roots \textit{strictly interlace} if we require all of the above inequalities to be strict. 
  \end{definition}

  %%%%%%%%%%%%%%%%%%%%%%%%%
  % Beginning of Discussion of Hermite Polynomials
  %%%%%%%%%%%%%%%%%%%%%%%%%%
  
  \section{\texorpdfstring{Differential Operators of the Form $\sum_{k=0}^{\infty} \frac{\gamma_k}{k!}D^k$}{Differential Operators} } \label{section:constant_coeffs}

  \label{sec:differential-operators}
  
  In this section, we will prove that the only linear operators $T \colon \R[x] \to \R[x]$ of the form $\sum_{k=0}^{\infty} \frac{\gamma_k}{k!}D^k$ such that $T[x^n]=P_n(x)$, where $\{P_n(x)\}$ is a real OPS, are essentially associated with Hermite polynomials.
  
  The physicist Hermite polynomials are orthogonal on $(-\infty,\infty)$ relative to the weight function $e^{-x^2}$ (see~\cite[Table 18.3.1]{NIST-handbook-2010}, \cite[p.145]{Chihara-Orthogonal-Polynomials-1978}, \cite[\S8.95]{Gradshteyn-Ryzhik-2000}) and may be represented as
  \begin{equation} \label{eqn:Physicist-Hermite-Differential-Operator}
    H_n(x) = 2^n e^{-\frac{D^2}{4}}x^n = (-1)^n e^{x^2} \tfrac{d^n}{dx^n}e^{-x^2}
    =  \sum_{m=0}^{\left \lfloor \frac{n}{2} \right \rfloor}
    \frac{n!(-1)^m}{m!(n-2m)!}(2x)^{n-2m}.
  \end{equation}
  
  The generalized probabilist Hermite polynomials are the monic polynomials that are orthogonal on the interval $(-\infty,\infty)$ relative to the probability density function $\tfrac{1}{\sqrt{2 \pi \alpha }}e^{-x^2/(2\alpha)}$ where $\alpha>0$ and can be expressed as follows (see \cite[p.146]{Chihara-Orthogonal-Polynomials-1978}):
  \begin{equation} \label{eqn:generalized-Hermite-Rodrigues}
    \He_n^{\alpha}(x) = e^{-\tfrac{\alpha}{2}D^2}x^n
    =(-\alpha)^n e^{x^2/(2\alpha)} \tfrac{d^n}{dx^n} e^{-x^2/(2\alpha)}
    = \sum_{k=0}^{\lfloor n/2 \rfloor} \frac{n!(-\alpha/2)^k}{k!(n-2k)!}x^{n-2k}.
  \end{equation}
  Furthermore, $\He_n^{\alpha}(x)$ satisfies the three-term recurrence formula
  \begin{equation} \label{eqn:Hermite-three-term}
    \He_n^{\alpha}(x) = x \He_{n-1}^{\alpha}(x)-\alpha(n-1)\He_{n-2}^{\alpha}(x).
  \end{equation}
  The physicist Hermite polynomials $H_n(x)$ are related to $\He^{\alpha}_n(x)$ by
  \begin{equation}
    H_n(x) = 2^n \He_n^{1/2}(x).
  \end{equation}
  
  From the Weierstrass product representation~\eqref{LP_Weirstrass}, we see that $\phi(z) = e^{-\frac{\alpha}{2}z^2} \in \LP$ for $\alpha > 0$. Therefore, Proposition~\ref{thm:LP-real-rootedness} implies that the two linear transformations $x^n \mapsto H_n(x)$ and $x^n \mapsto \He_n^{\alpha}(x)$ are real-root preserving transformations.
  
  The operator $\exp(-\tfrac{\alpha}{2}D^2)$ can be written as $\sum_{k=0}^\infty \tfrac{(-1)^k(\alpha/2)^k}{k!}D^{2k}$. Thus, in terms of equation~\eqref{eqn:differential-operator-representation}, all of the $Q_k(x)$ are constants. This raises the question: ``Are there other transformations where $\{T[x^n]\}$ is an OPS and each $Q_k(x)$ is constant?" Theorem~\ref{thm:equivalent-conditions-Hermite} below states that the answer is no, up to a linear shift of the polynomials. As seen from the following theorem, this question is intimately tied to the study of Appell sequences. 
  \begin{proposition}[\cite{Shohat-1936}, Thm.~IV] \label{thm:Shohat}
    The only real OPS which is at the same time an Appell sequence is orthogonal on $(-\infty,\infty)$ relative to a weight function of the form $e^{-h^2(x-c)^2}$ ($h,c$ constants)  and is therefore reducible to the probabilist Hermite polynomials by a linear transformation.
  \end{proposition}
  
  \begin{remark}
    Equation~\eqref{eqn:differential-operator-LaguerreAppell-Sequence} below combined with Proposition~\ref{thm:Shohat} imply that the only types of OPS that are of the form $\{\phi(D)x^n\}$ where $\phi(x) = \sum_{k=0}^\infty \frac{\gamma_k}{k!} x^k$ are linear shifts of generalized probabilist Hermite polynomials. This is the content of the main theorem of this section. The proof of parts $\ref{itm:Appell}\Rightarrow\ref{itm:recur}\Leftrightarrow\ref{itm:shift}$ follow the methods of Shohat~\cite{Shohat-1936}. The rest of the theorem highlights how this theory fits into our study of the differential operator representations.
  \end{remark}
  
  \begin{theorem} \label{thm:equivalent-conditions-Hermite}
    Let $T$ be a linear transformation such that $\{T[x^n]\}=\{P_n(x)\}$ forms a real OPS, where $\deg P_n(x)=n$ for all $n \geq 0$. The following are equivalent:
    \begin{enumerate} [label=\rm(\roman{*}), ref=\rm(\roman{*})]  \itemsep=3pt
      \item \label{itm:const} $T = \sum_{k=0}^\infty \frac{\gamma_k}{k!} D^k$, where $\gamma_k \in \R$ for all $k$.
      \item \label{itm:Appell}$\{P_n(x)\}$ is an Appell sequence.
      \item \label{itm:recur} For some constants $\beta \in \R, \alpha> 0,$ $\{P_n(x)\}$ satisfies a recurrence relation of the form
      \begin{equation*} % \label{eqn:differential-operator-Laguerrerecurrence-part3}
        P_n(x)=(x - \beta)P_{n-1}(x)- \alpha(n - 1) P_{n-2}(x),
      \end{equation*}
      where  $P_{-1}(x) = 0$ and $P_0(x) = \gamma_0.$
      \item \label{itm:shift}
      For all $n$, $P_n(x) = \gamma_0 \He_n^{\alpha}(x-\beta)$ with $\gamma_0 \neq 0, \beta \in \R,$ and $\alpha > 0$.
      \item \label{itm:DO}
      $T = \gamma_0 e^{-\frac{\alpha}{2}D^2-\beta D}$, where $\alpha>0$ and $\beta \in \R$.
    \end{enumerate}
  \end{theorem}
  
  \begin{proof}
    \noindent \ref{itm:const} $\Rightarrow$ \ref{itm:Appell}: Assume~\ref{itm:const}. Then,
    \begin{align} \label{eqn:expression-for-P_n}
      P_n(x) & =\Big(\sum\limits_{k=0}^\infty \frac{\gamma_k}{k!}D^k\Big)[x^n]=\sum\limits_{k=0}^n
      \frac{\gamma_k}{k!} n(n-1) \cdots  (n-k+1)x^{n-k} \nonumber                                  \\
      & =
      \sum\limits_{k=0}^n \gamma_k \mybinom{n}{n-k} x^{n-k}
      =
      \sum\limits_{k=0}^n \gamma_{n-k} \mybinom{n}{k} x^k \text{ for } n \geq 0.
    \end{align}
    Note that the leading coefficient of $P_n(x)$ is $\gamma_0$ for each $n$. Furthermore, since $\deg P_n(x)= n$ for all $n \geq 0$, the constant $\gamma_0$ is nonzero. Differentiating this expression yields
    \begin{align} \label{eqn:differential-operator-LaguerreAppell-Sequence}
      P_n'(x) & = \sum_{k=1}^n \gamma_{n-k} \mybinom{n}{k} k x^{k-1}
      = n\sum_{k=1}^n \gamma_{n-k} \mybinom{n-1}{k-1} x^{k-1} \nonumber \\
      & =
      n\sum_{k=0}^{n-1}\gamma_{n-1-k} \mybinom{n-1}{k} x^k = nP_{n-1}(x).
    \end{align}
    Hence, $\{P_n(x)\}$ is an Appell sequence.
    
    \medskip \noindent \ref{itm:Appell}$\Rightarrow$\ref{itm:recur}:
    By the standard theory for orthogonal polynomials, $\{P_n(x)\}$ satisfies a three term recurrence relation of the form:
    \begin{equation} \label{eqn:3termrecurrence}
      P_n(x) = (x-c_n)P_{n-1}(x) - \lambda_n P_{n-2}(x)
    \end{equation}
    where $c_n \in \R$ and  $\lambda_n >0$.
    
    A simple induction argument using anti-differentiation shows that if $\{P_n(x)\}_{n=0}^{\infty}$ is an Appell sequence, then there exists a sequence $\{\gamma_n\}_{n=0}^{\infty}$ such that
    \[
    P_n(x) = \sum_{k=0}^{n} \gamma_{n-k} \mybinom{n}{k} x^k
    \]
    for all $n \geq 0$.
    
    Substituting this representation of $P_n(x)$ into Equation~\ref{eqn:3termrecurrence} gives
    \begin{align*}
      \sum_{k=0}^{n} \gamma_{n-k}\mybinom{n}{k} x^k
      & =
      (x-c_n)\sum_{k=0}^{n-1} \gamma_{n-1-k}\mybinom{n-1}{k} x^k
      -\lambda_n \sum_{k=0}^{n-2} \gamma_{n-2-k}\mybinom{n-2}{k} x^k.
    \end{align*}
    Equating the coefficients of $x^k$ in the various terms and using the convention $\gamma_m=0$ for $m<0$, gives for $n \geq 0$
    \begin{equation*}
      \gamma_{n-k}\mybinom{n}{k}
      =
      \gamma_{n-k}\mybinom{n-1}{k-1} - c_n \gamma_{n-1-k}\mybinom{n-1}{k}
      - \lambda_n \gamma_{n-2-k}\mybinom{n-2}{k}.
    \end{equation*}
    Consequently, for $n \geq 1$,
    \begin{equation}  \label{eqn:differential-operator-Laguerrecoefficient-equality}
      \gamma_{n-k}\mybinom{n-1}{k}
      =
      -c_n \gamma_{n-1-k} \mybinom{n-1}{k}
      -\lambda_n \gamma_{n-2-k}\mybinom{n-2}{k}.
    \end{equation}
    
    Setting $k=n-1$ in equation~\eqref{eqn:differential-operator-Laguerrecoefficient-equality} gives
    \begin{align*}
      \gamma_1 \mybinom{n-1}{n-1} & = -c_n \gamma_0 \mybinom{n-1}{n-1} - 0 \\
      c_n                         & = -\frac{\gamma_1}{\gamma_0} .
    \end{align*}
    Thus, $c_n$ is constant. Setting $k=n-2$, where $n \geq 2$, in  equation~\eqref{eqn:differential-operator-Laguerrecoefficient-equality} gives
    \[
    \gamma_2 \mybinom{n-1}{n-2}      = -c_n \gamma_1  \mybinom{n-1}{n-2} -\lambda_n \gamma_0 \mybinom{n-2}{n-2}
    \]
    Substituting for $c_n$ and solving for $\lambda_n$ results in
    \[
    \lambda_n = (n-1) \bigg(\frac{\gamma_1^2-\gamma_0 \gamma_2}{\gamma_0^2}\bigg).
    \]
    Setting $\beta = -\tfrac{\gamma_1}{\gamma_0}$ and $\alpha = \tfrac{\gamma_1^2-\gamma_0 \gamma_2}{\gamma_0^2}$ proves~\ref{itm:recur}.
    
    \medskip \noindent \ref{itm:recur}$\Leftrightarrow$\ref{itm:shift}:
    From Chihara~\cite[p.~108]{Chihara-Orthogonal-Polynomials-1978}, we know that if $P_n(x)$ satisfies the recurrence relation~\eqref{eqn:3termrecurrence}, then the shifted polynomials $R_n(x) = P_n(x + s)$ satisfy the recurrence
    \[
    R_n(x)=(x-(c_n-s))R_{n-1}(x)-\lambda_n R_{n-2}(x), \quad n \geq 1.
    \]
    Then, the three term recurrence for the probabilist Hermite polynomials (Equation~\eqref{eqn:Hermite-three-term}) proves \ref{itm:recur}$\Leftrightarrow$\ref{itm:shift}.
    
    \medskip \noindent \ref{itm:shift}$\Rightarrow$\ref{itm:DO}:
    For a function $f(x)$ in the Laguerre-P\'olya class, $e^{-\beta D}f(x) = f(x-\beta)$. By Equation~\eqref{eqn:generalized-Hermite-Rodrigues}, $e^{-\frac{\alpha}{2}D^2}x^n = \He_n^{\alpha}(x)$, so $\gamma_0 \He_n^{\alpha}(x-\beta) = \gamma_0 e^{-\frac{\alpha}{2}D^2-\beta D}x^n$ for all $n$.
    
    \medskip \noindent \ref{itm:DO}$\Rightarrow$\ref{itm:const} is immediate.
  \end{proof}
  
  In Theorem~\ref{thm:Hermite-linear-tranformation-Qk}, we will see that if $\{H_n(x)\}$ is the set of physicist Hermite polynomials and if $T[x^n]=H_n(x)$, then unlike in the case of the probabilist Hermite polynomials, in the representation
  \[
  T=\sum_{k=0}^{\infty} \frac{Q^H_k(x)}{k!} D^k
  \]
  the polynomial $Q^H_k(x)$ has degree $k$ rather than being a constant. Remarkably, $Q^H_k(x)$ is in fact a rescaled Hermite polynomial.

  \section{\texorpdfstring{Reality and Interlacing of the Roots of the $Q^H_k(x)$}{Reality and Interlacing of the Roots}} \label{section:Hermite}
  Consider the set of physicist Hermite polynomials defined by~\eqref{eqn:Physicist-Hermite-Differential-Operator}. Let $\leftidx{^+\!}{Q^H_k(x)}{}$ be the coefficient polynomials for the transformation $x^n \mapsto H_n(x)$ and let $\leftidx{^-\!}{Q^H_k(x)}{}$ be the coefficient polynomials for the transformation $x^n \mapsto H_n(-x)$.
  Observe that in the transformation
  $x^n \mapsto H_n(x)$, we have
  \[
  x-i \mapsto (2x)-i,
  \]
  and both polynomials have a single root in the upper half plane. By Proposition~\ref{Borcea-Branden-prop}, this transformation is stability preserving. Therefore, we cannot apply Proposition~\ref{prop:Chasse-stability-reversing} to conclude that $Q_k^{H}(x)$ has only real roots. However, the transformation $x \mapsto H_n(-x)$ is stability reversing, and therefore the  $\leftidx{^-\!}{Q^H_k(x)}{}$ has only real roots. The following theorem shows that, in fact, both of the sequences $\{\leftidx{^+\!}{Q^H_k(x)}{}\}$ and $\{\leftidx{^-\!}{Q^H_k(x)}{}\}$ have real, interlacing roots.
  
  \begin{theorem} \label{thm:Hermite-linear-tranformation-Qk}
    Let $\beta$ be a nonzero real number. Let $T_{\beta}$ be the linear operator defined by $T_{\beta}[x^n] = H_n(\frac{\beta}{2}x)$. Then,
    \[
    T_{\beta} = \sum_{n=0}^\infty H_n\big(\tfrac{\beta-1}{2}x\big)\frac{D^n}{n!}.
    \]
    Since the Hermite polynomials have real interlacing roots, this shows that the coefficient polynomials for the transformation $x^n \mapsto H_n(x)$ have real interlacing roots.
  \end{theorem}
  
  \begin{proof}
    From equation~\eqref{eqn:Physicist-Hermite-Differential-Operator}, the $n$\textsuperscript{th} physicist Hermite polynomial is
    \begin{equation}
      H_{n}(x) = n!\sum_{k=0}^{\lfloor n/2\rfloor}\frac{(-1)^{k}}{k!(n-2k)!}(2x)^{n-2k}.
    \end{equation}
    By Equation~\eqref{eqn:Q-general-expression}, the coefficient polynomials for the transformation $T_\beta$ are
    \begin{align*}
      & \;\;\;\;\sum_{k=0}^{n}\mybinom{n}{k}T[x^{k}](-x)^{n-k}                                                                                   \\
      & =\sum_{k=0}^{n}\mybinom{n}{k}\bigg[k!\sum_{j=0}^{\lfloor k/2\rfloor}\frac{(-1)^{j}}{j!(k-2j)!}(\beta x)^{k-2j}\bigg](-x)^{n-k}           \\
      & =\sum_{j=0}^{\lfloor n/2\rfloor}\bigg[\sum_{k=2j}^{n}\frac{n!(-1)^{j+n+k}}{(n-k)!j!(k-2j)!}\beta^{k-2j}\bigg]x^{n-2j}                    \\
      & =\sum_{j=0}^{\lfloor n/2\rfloor}\bigg[\sum_{k=0}^{n-2j}\frac{n!(-1)^{j+n+k}}{(n-2j-k)!j!k!}\beta^{k}\bigg]x^{n-2j}                       \\
      & =\sum_{j=0}^{\lfloor n/2\rfloor}\bigg[\frac{n!(-1)^{j}}{j!(n-2j)!}\sum_{k=0}^{n-2j}\mybinom{n-2j}{k}\beta^{k}(-1)^{n-2j-k}\bigg]x^{n-2j} \\
      & =\sum_{j=0}^{\lfloor n/2\rfloor}\bigg[\frac{n!(-1)^{j}}{j!(n-2j)!}(\beta-1)^{n-2j}\bigg]x^{n-2j}                                         \\
      & =n!\sum_{j=0}^{\lfloor n/2\rfloor}\frac{(-1)^{j}}{j!(n-2j)!}\big((\beta-1)x\big)^{n-2j}                                                  \\
      & =H_{n}(\tfrac{\beta-1}{2}x). \qedhere
    \end{align*}
  \end{proof}
  
  In light of our discussion at the beginning of the section, the following corollary is an immediate consequence of Theorem~\ref{thm:Hermite-linear-tranformation-Qk}.
  \begin{corollary}   \label{thm:stability-preserving-stability-reversing}
    For all $k$,
    $\leftidx{^-\!}{Q^H_k(x)}{} = \leftidx{^+\!}{Q^H_k(-3x)}{}$.
  \end{corollary}

  \section{\texorpdfstring{Reality and Interlacing of the Roots of the $Q^{L}_k(x)$}{Reality and Interlacing of the Roots}} \label{section:Laguerre}
  
  The generalized Laguerre polynomials $L_n^{\alpha}(x)$ for the real parameter $\alpha$ are orthogonal on the interval $(0,\infty)$ relative to the weight function $e^{-x}x^{\alpha}$ where $\alpha>-1$. As all of the results in this paper are for the case $\alpha=0$, we will use the simpler notation $L_n(x) = L_n^0(x)$, which is standard in the literature (see \cite[p.145]{Chihara-Orthogonal-Polynomials-1978}, \cite[p.28]{Rahman-Schmeisser-2002}). The polynomials $L_n(x)$ are often called the standard Laguerre polynomials.
  
  The polynomials $L_n(x)$ satisfy the well-known closed form formula~\cite[p.201]{Rainville-1960}
  \begin{equation} \label{eqn:Ln-closed-form}
    L_{n}(x) =\sum_{k=0}^{n}\mybinom{n}{k}\frac{(-x)^{k}}{k!}
  \end{equation}
  as well as the differential-difference equation~\cite[p.202]{Rainville-1960}
  \begin{equation} \label{eqn:differential-difference-for-LaguerreL}
    L_{n+1}(x) = \frac{1}{n+1} \big( x L'_n(x) + (n+1-x)L_n(x) \big).
  \end{equation}
  Fisk~\cite{Fisk-Laguerre-Polynomials-2008} proved that the transformation $x^n \mapsto L_n(x)$ preserves real-rootedness. By a similar computation as in the previous section, one can see that this transformation is stability reversing. By Proposition~\ref{prop:Chasse-stability-reversing}, the coefficient polynomials $Q_k^{L}(x)$ have real roots. In this section, we give an explicit formula for these coefficient polynomials and then use this formula to show that the roots are simple and lie in the interval $(0,1)$. We also show that the roots of $Q_k^{L}(x)$ and $Q_{k+1}^{L}(x)$ strictly interlace.
  
  %%%%%%%%%%%%%%%%%%%%%%
  \begin{theorem} \label{thm:Laguerre-differential-operator}
    The differential operator representation of the Laguerre linear transformation $T^{L} \colon x^n \mapsto L_n(x)$ is given by
    \begin{equation} \label{eqn:differential-operator-Laguerre}
      T^{L}=\sum_{k=0}^{\infty} \frac{ Q^{L}_k(x)  }{k!}D^k,
    \end{equation}
    where
    \begin{equation} \label{eqn:closed-form-expression-for-Laguerre-Qk}
      Q^{L}_k(x)= \sum_{r=0}^k \bigg[ \mybinom{k}{r} \sum_{\ell=0}^{r} \mybinom{r}{\ell} \frac{1}{\ell!}\bigg] (-x)^r \quad\text{ for all } k.
    \end{equation}
  \end{theorem}
  
  \begin{proof}
    By equations~\eqref{eqn:Ln-closed-form} and~\eqref{eqn:Q-general-expression},
    \begin{align*}
      Q_{n}^{L}(x) & =\sum_{k=0}^{n}\binom{n}{k}L_{n-k}(x)(-x)^{k}                                                        \\
      & =\sum_{k=0}^{n}\binom{n}{k}\bigg[\sum_{j=0}^{n-k}\binom{n-k}{n-k-j}\frac{(-x)^{j}}{j!}\bigg](-x)^{k} \\
      & =\sum_{r=0}^{n}\bigg[\sum_{j=0}^{r}\binom{n}{r-j}\binom{n-r+j}{n-r}\frac{1}{j!}\bigg](-x)^{r}.       \\
      & = \sum_{r=0}^{n}\bigg[\binom{n}{r}\sum_{j=0}^{r}\binom{r}{j}\frac{1}{j!}\bigg](-x)^{r}.  \qedhere
    \end{align*}
  \end{proof}
  
  Now we consider the Laguerre linear transformation $x^n \mapsto L_n(x)$ with coefficient polynomials as in equation~\eqref{eqn:closed-form-expression-for-Laguerre-Qk}. Unlike the polynomials $Q_k^H(x)$ associated with the transformation $x^n \to H_n(x)$, the coefficient polynomials $Q_k^{L}(x)$ are not orthogonal with respect to any moment functional because they do not satisfy a three-term recurrence relation. For example, there do not exist constants $a,b,c$ such that $Q_3^L(x)=(ax+b)Q^L_2(x)+cQ^L_1(x)$. However, these polynomials have interlacing roots in the interval $(0,1)$. To prove this, we modify arguments of Dominici, Driver, and Jordaan~\cite{Dominici-Driver-Jordaan-2011}, who showed interlacing of roots when the sequence of polynomials satisfies a specific class of differential-difference equations. The polynomials $Q_k^{L}(x)$ also satisfy a differential-difference equations as seen in the following lemma.
  \begin{lemma} \label{lemma:differential-difference-equation-for-Laguerre-Qk}
    The coefficient polynomials $Q^{L}_n(x)$ %corresponding to the transformation $T[x^n] = L_n(x)$ given in equation
    \eqref{eqn:closed-form-expression-for-Laguerre-Qk} satisfy the following differential-difference equation for all $n$.
    \begin{equation} \label{eqn:differential-difference-equation-for-Laguerre-Qk}
      (n+1) Q^{L}_{n+1}(x) = (x-2x^2+x^3)(Q^{L}_n)'(x) + (n+1-2x -nx^2)Q^{L}_n(x).
    \end{equation}
  \end{lemma}
  \begin{proof} Theorem~\ref{thm:Laguerre-differential-operator} gives an explicit representation of $Q_n^{L}(x)$ as follows:
    \begin{equation}
      Q^{L}_{n}(x)=\sum_{r=0}^{n}\bigg[\sum_{\ell=0}^{r}\mybinom{n}{r}\mybinom{r}{\ell}\frac{1}{\ell!}\bigg](-x)^{r}.\label{eqn:Qn-closed-form}
    \end{equation}
    Next we calculate each term on the right hand side of equation~\eqref{eqn:differential-difference-equation-for-Laguerre-Qk}.
    
    \begin{align}
      x(Q^{L}_{n})'(x)       & =\sum_{r=0}^{n+2}\bigg[\sum_{\ell=0}^{r}r\tbinom{n}{r}\tbinom{r}{\ell}\tfrac{1}{\ell!}\bigg](-x)^{r}, \label{eqn:RHS-a}          \\
      -2x^{2}(Q^{L}_{n})'(x) & =\sum_{r=0}^{n+2}\bigg[2(r-1)\sum_{\ell=0}^{r}\tbinom{n}{r-1}\tbinom{r-1}{\ell}\tfrac{1}{\ell!}\bigg](-x)^{r}, \label{eqn:RHS-b} \\
      x^{3}(Q^{L}_{n})'(x)   & =\sum_{r=0}^{n+2}\bigg[(r-2)\sum_{\ell=0}^{r}\tbinom{n}{r-2}\tbinom{r-2}{\ell}\tfrac{1}{\ell!}\bigg](-x)^{r}, \label{eqn:RHS-c}  \\
      (n+1)Q^{L}_{n}(x)      & =\sum_{r=0}^{n+2}\bigg[(n+1)\sum_{\ell=0}^{r}\tbinom{n}{r}\tbinom{r}{\ell}\tfrac{1}{\ell!}\bigg](-x)^{r}, \label{eqn:RHS-d}      \\
      -2xQ^{L}_{n}(x)        & =\sum_{r=0}^{n+2}\bigg[2\sum_{\ell=0}^{r}\tbinom{n}{r-1} \tbinom{r-1}{\ell}\tfrac{1}{\ell!}\bigg](-x)^{r}, \label{eqn:RHS-e}     \\
      -nx^{2}Q^{L}_{n}(x)    & =\sum_{r=0}^{n+2}\bigg[n\sum_{\ell=0}^{r}\tbinom{n}{r-2}\tbinom{r-2}{\ell}\tfrac{1}{\ell!}\bigg](-x)^{r}. \label{eqn:RHS-f}
    \end{align}
    
    We require three simple binomial coefficient identities:
    \begin{align}
      \mybinom{n}{r}\mybinom{r}{\ell}     & =\frac{n-r+1}{n+1}\mybinom{n+1}{r}\mybinom{r}{\ell}\label{eqn:Binomial-a}                        \\
      \mybinom{n}{r-1}\mybinom{r-1}{\ell} & =\frac{r-\ell}{n+1}\mybinom{n+1}{r}\mybinom{r}{\ell}\label{eqn:Binomial-b}                       \\
      \mybinom{n}{r-2}\mybinom{r-2}{\ell} & =\frac{(r-\ell)(r-\ell-1)}{(n+1)(n-r+2)}\mybinom{n+1}{r}\mybinom{r}{\ell}.\label{eqn:Binomial-c}
    \end{align}
    
    Then using equations~\eqref{eqn:RHS-a} through~\eqref{eqn:RHS-f} and the identities in equations~\eqref{eqn:Binomial-a} through~\eqref{eqn:Binomial-c}, we find that the coefficient of $(-x)^{r}$ on the right hand side of equation~\eqref{eqn:differential-difference-equation-for-Laguerre-Qk} is
    \begin{align*}
      & (n+1+r)\sum_{\ell=0}^{r}\tbinom{n}{r}\tbinom{r}{\ell}\tfrac{1}{\ell!}
      +2r\sum_{\ell=0}^{r}\tbinom{n}{r-1}\tbinom{r-1}{\ell}\tfrac{1}{\ell!}
      +(-n+r-2)\sum_{\ell=0}^{r}\tbinom{n}{r-2}\tbinom{r-2}{\ell}\tfrac{1}{\ell!}                                                                                                                                                                      \\
      & =
      (n+1+r)\sum_{\ell=0}^{r}\tfrac{n-r+1}{n+1}\tbinom{n+1}{r}\tbinom{r}{\ell}\tfrac{1}{\ell!}
      +2r\sum_{\ell=0}^{r}\tfrac{r-\ell}{n+1}\tbinom{n+1}{r}\tbinom{r}{\ell}\tfrac{1}{\ell!}                                                                                                                                                           \\ \nobreak
      & \qquad\qquad + (-n+r-2)\sum_{\ell=0}^{r}\tfrac{(r-\ell)(r-\ell-1)}{(n+1)(n-r+2)}\tbinom{n+1}{r}\tbinom{r}{\ell}\tfrac{1}{\ell!}                                                                                                               \\
      & =\text{\ensuremath{\tfrac{1}{n+1}\sum_{\ell=0}^{r}}\ensuremath{\left[(n+1+r)(n-r+1)+2r(r-\ell)-(r-\ell)(r-\ell-1)\right]}}\tbinom{n+1}{r}\tbinom{r}{\ell}\tfrac{1}{\ell!}                                                                     \\
      & =\text{\ensuremath{\tfrac{1}{n+1}\sum_{\ell=0}^{r}}\ensuremath{\left[1+2n+n^{2}+r-\ell-\ell^{2}\right]}}\tbinom{n+1}{r}\tbinom{r}{\ell}\tfrac{1}{\ell!}                                                                                       \\
      & =(n+1)\sum_{\ell=0}^{r}\tbinom{n+1}{r}\tbinom{r}{\ell}\tfrac{1}{\ell!}+\tfrac{1}{n+1}\sum_{\ell=0}^{r}(r-\ell-\ell^{2})\tbinom{n+1}{r}\tbinom{r}{\ell}\tfrac{1}{\ell!}                                                                        \\
      & =(n+1)\sum_{\ell=0}^{r}\tbinom{n+1}{r}\tbinom{r}{\ell}\tfrac{1}{\ell!}+\tfrac{1}{n+1}\tbinom{n+1}{r}\bigg[\sum_{\ell=0}^{r}(r-\ell)\tbinom{r}{\ell}\tfrac{1}{\ell!}-\sum_{\ell=0}^{r}\ell^{2}\tbinom{r}{\ell}\tfrac{1}{\ell!}\bigg]           \\
      & =(n+1)\sum_{\ell=0}^{r}\tbinom{n+1}{r}\tbinom{r}{\ell}\tfrac{1}{\ell!}+\tfrac{1}{n+1}\tbinom{n+1}{r}\bigg[\sum_{\ell=0}^{r}(\ell+1)^{2}\tbinom{r}{\ell+1}\tfrac{1}{(\ell+1)!}-\sum_{\ell=0}^{r}\ell^{2}\tbinom{r}{\ell}\tfrac{1}{\ell!}\bigg] \\
      & =(n+1)\sum_{\ell=0}^{r}\tbinom{n+1}{r}\tbinom{r}{\ell}\tfrac{1}{\ell!}+\tfrac{1}{n+1}\tbinom{n+1}{r}\bigg[\sum_{\ell=0}^{r}\ell{}^{2}\tbinom{r}{\ell}\tfrac{1}{\ell!}-\sum_{\ell=0}^{r}\ell^{2}\tbinom{r}{\ell}\tfrac{1}{\ell!}\bigg]         \\
      & =(n+1)\sum_{\ell=0}^{r}\tbinom{n+1}{r}\tbinom{r}{\ell}\tfrac{1}{\ell!}   .
    \end{align*}
    The last expression is the coefficient of $(-x)^{r}$ on the left hand side of equation~\eqref{eqn:differential-difference-equation-for-Laguerre-Qk}. This proves that equation~\eqref{eqn:differential-difference-equation-for-Laguerre-Qk} holds.
  \end{proof}
  
  \begin{theorem} \label{thm:Laguerre-Q-interlacing}
    The roots of the $Q^{L}_n(x)$ given in equation~\eqref{eqn:closed-form-expression-for-Laguerre-Qk} all lie in the interval $(0,1)$. Furthermore, the roots of $Q^{L}_n(x)$ and $Q^{L}_{n+1}(x)$ strictly interlace for $n \geq 1$.
  \end{theorem}
  
  \begin{proof}
    As noted before, this proof modifies the arguments given in~\cite{Dominici-Driver-Jordaan-2011}.  Set
    \[
    A_n(x) = \frac{1}{n+1}(x-2x^2+x^3) \quad \text{ and } \quad
    B_n(x) = \frac{1}{n+1}(n+1-2x -nx^2),
    \]
    so that by Lemma~\ref{lemma:differential-difference-equation-for-Laguerre-Qk},
    \[
    Q^{L}_{n+1}(x) = A_n(x)(Q^{L}_n)'(x) + B_n(x)Q^{L}_n(x).
    \]
    From the explicit representation of $Q_n^{L}(x)$ in equation~\eqref{eqn:closed-form-expression-for-Laguerre-Qk} we have $Q_0^{L}(x)=1$, which has no roots, and $Q_1^{L}(x)=1-2x$ whose sole root $1/2$ belongs to the interval $(0,1)$. Now suppose by way of induction that $Q_n^{L}(x)$ has simple real roots $r_1<r_2<\cdots<r_n$ in the interval $(0,1)$. We will consider the rational function
    \begin{align*}
      F_{n+1}(x) & = \frac{Q_{n+1}^{L}(x)}{A_n(x)Q_n^{L}(x)} = \frac{A_n(x)(Q_n^{L})'(x)+B_n(x)Q_n^{L}(x)}{A_n(x)Q_n^{L}(x)} \\
      & = \frac{B_n(x)}{A_n(x)} + \frac{(Q_n^{L})'(x)}{Q_n^{L}(x)}
      = \frac{n+1-2x-nx^2}{x(x-1)^2} + \frac{(Q_n^{L})'(x)}{Q_n^{L}(x)}                                                      \\
      & = \frac{n+1}{x}-\frac{1}{(x-1)^2}-\frac{2 n+1}{x-1}
      +\frac{1}{x-r_1}+\cdots+\frac{1}{x-r_n}.
    \end{align*}
    $F_{n+1}(x)$ is positive immediately to the right of $0$, negative immediately to the left of each $r_k$, positive immediately to the right of each $r_k$, and negative immediately to the left of $1$.  By the Intermediate Value Theorem, each of the intervals
    \[
    (0,r_1), \; (r_1,r_2), \; \ldots \; (r_{n-1}, r_n), \; (r_n,1)
    \]
    contains a root of $F_n(x)$ which gives $n+1$ simple roots of the $(n+1)$\textsuperscript{st} degree polynomial $Q_{n+1}^{L}(x)$ in the interval $(0,1)$, interlacing with the $n$ roots of $Q_n^{L}(x)$.
  \end{proof}
  
  \section{A generating function for the coefficient polynomials of the Laguerre transformation} \label{section:generating function}
  
  It is well-known (see~\cite[p.101]{Szego-1939-orthogonal-polynomials}) that the Laguerre polynomials $\{L_n(x)\}$ have a simple generating function
  \begin{equation} \label{eqn:Laguerre-L-generating-function}
    F(x,w) = \frac{\exp\left( \frac{wx}{w-1} \right)}{1-w}
    = \sum_{n=0}^{\infty} L_n(x) w^n.
  \end{equation}
  We naturally ask whether the coefficient polynomials $\{Q_n^{L}(x)\}$ associated with the Laguerre transformation $T \colon x^n \to L_n(x)$ have a generating function as simple as the one in Equation~\eqref{eqn:Laguerre-L-generating-function}. This is the case, as described in the next theorem:
  
  \begin{theorem} \label{thm:coefficient-polynomials-generating-function}
    If $T \colon x^n \to L_n(x)$ is the Laguerre linear transformation with coefficient polynomials $\{Q_n^{L}(x)\}$ as in Theorem~\ref{thm:Laguerre-differential-operator}, then we have
    \begin{equation} \label{eqn:Laguerre-Q-generating-function}
      \frac{\exp\left(-\frac{w x}{w (x-1)+1}\right)}{w (x-1)+1} = \sum_{n=0}^{\infty} Q_n^{L}(x) w^n.
    \end{equation}
  \end{theorem}
  
  \begin{proof}
    Let $G(x,w) = \sum_{n=0}^{\infty} Q_n^{L}(x) w^n$ be a generating function for $\{Q_n^{L}(x)\}$. We will show that $G(x,w)$ satisfies the partial differential equation
    \begin{equation} \label{eqn:Qk-PDE}
      x(x-1)^2 \frac{\pd G}{\pd x} +(1-x^2) w \frac{\pd G}{\pd w} +(1-2x)G = \frac{\pd G}{\pd w}.
    \end{equation}
    We take advantage of the differential-difference equation in Lemma~\ref{lemma:differential-difference-equation-for-Laguerre-Qk}.
    \begin{equation} \label{eqn:generating-function-proof-1}
      x(x-1)^2 \frac{\pd G}{\pd x}
      = \sum_{n=0}^{\infty} x(x-1)^2 (Q_n^{L})'(x) w^n
      = \sum_{n=0}^{\infty} (n+1)A_n(x) (Q_n^{L})'(x) w^n.
    \end{equation}
    Then
    \begin{equation} \label{eqn:generating-function-proof-2}
      (1-x^2) w \frac{\pd G}{\pd w}
      = \sum_{n=0}^{\infty} n(1-x^2) Q_n(x)w^n
    \end{equation}
    and
    \begin{equation} \label{eqn:generating-function-proof-3}
      (1-2x) G = \sum_{n=0}^{\infty} (1-2x) Q_n^{L}(x) w^n.
    \end{equation}
    Adding Equations~\eqref{eqn:generating-function-proof-2} and~\eqref{eqn:generating-function-proof-3} gives
    \begin{equation} \label{eqn:generating-function-proof-4}
      \begin{split}
        (1-x^2)w \frac{\pd G}{\pd w} + (1-2x) G
        & =
        \sum_{n=0}^{\infty} \big( 1-2x + n(1-x^2)\big) Q_n(x) w^n \\
        & = \sum_{n=0}^{\infty} (n+1)B_n(x) Q_n(x) w^n.
      \end{split}
    \end{equation}
    Adding Equations~\eqref{eqn:generating-function-proof-1} and~\eqref{eqn:generating-function-proof-4} gives
    \begin{align*}
      & x(x-1)^2 \frac{\pd G}{\pd x} + (1-x^2)w\frac{\pd G}{\pd w} + (1-2x) G           \\
      & = \sum_{n=0}^{\infty} (n+1)\big( A_n(x) (Q_n^{L})'(x) + B_n(x) Q_n(x) \big) w^n \\
      & = \sum_{n=0}^{\infty} (n+1) Q_{n+1}^{L}(x) w^n
      = \sum_{n=1}^{\infty} n Q_n^{L}(x) w^{n-1}
      = \frac{\pd G}{\pd w},
    \end{align*}
    which proves Equation~\eqref{eqn:Qk-PDE}. The generating function satisfies the initial condition $G(x,0)=Q_0^{L}(x)=1$. One easily verifies that the function
    \[
    G(x,w)=\frac{\exp\left(-\frac{w x}{w (x-1)+1}\right)}{w (x-1)+1}
    \]
    satisfies Equation~\eqref{eqn:Qk-PDE} and $G(x,0)=1$. This completes the proof of the theorem.
  \end{proof}
  
  \section{Chebyshev and Legendre Polynomials} \label{section:ChebLeg}
  Because we have discussed the classical Hermite and Laguerre polynomials, it is natural to mention the classical Chebyshev and Legendre polynomials as well. In these two cases, the coefficient polynomials are all multiples of $x^2 - 1.$
  
  The Chebyshev polynomials $P_n^T$ are the unique polynomials satisfying $P^T_n(\cos \theta) = \cos(n\theta)$. A closed form representation \cite[Eqn. 8.940, p.983]{Gradshteyn-Ryzhik-2000} is given by
  \[
  P^T_n(x) = \sum_{k=0}^{n} \mybinom{n}{2k}(x^2-1)^k x^{n-2k}.
  \]
  The transformation $T[x^n] = P^T_n(x)$ does not preserve the reality of zeros. However, we do have the following Proposition due to Iserles and Saff.
  \begin{proposition}[\cite{Iserles-Saff-1989}, Thm.~1, p.~559] \label{prop:Iserles-Saff}
    If the polynomial $\sum_{k=0}^n a_k x^k$ with real coefficients has all of its zeros in the complex open unit disk, then all of the zeros of $\sum_{k=0}^n a_k P^T_k(x)$ lie in the open interval $(-1,1)$.
  \end{proposition}
  Chasse proved the following:
  \begin{proposition}[\cite{Chasse-PhD-2011}, Chasse, Prop.~214]
    If $T=\sum_{n=0}^{\infty} \frac{Q_n(x)}{n!} D^n$ is the monomial to Chebyshev polynomial basis transformation, defined by $T[x^n]=P_k^T(x)$, then the coefficients $Q^T_n(x)$, when nonzero, have only real zeros. Furthermore,
    \[
    Q^T_{2k}(x) = (x^2-1)^k, \qquad Q_{2k+1}^T(x)=0, \qquad k\ge 0.
    \]
  \end{proposition}

  The Legendre polynomials $\{P_n^{Le}(x)\}$ are related to the Chebyshev polynomials. They are both special cases of the more general Jacobi polynomials. The Legendre polynomials can be expressed in terms of a Rodrigues formula $P^{Le}_n(x) = {\frac{1}{2^n n!}\frac{d^n}{dx^n}(x^2-1)^n}$ and also in closed form \cite[Eqn. 2.6, p.144]{Chihara-Orthogonal-Polynomials-1978} as
  \[
  P^{Le}_n(x) = \frac{1}{2^n}\sum_{k=0}^{n} \mybinom{n}{k}^2 (x-1)^{n-k} (x+1)^k.
  \]
  Similar to the result for the Chebyshev polynomials, we have the following result, due to Chasse.
  \begin{proposition}[\cite{Chasse-2014}, Thm.~1.2, p.~2] \label{prop:Legendre-zeros}
    If $f(x) = \sum_{k=0}^n a_k x^k$ has all of its zeros in the
    interval $(-1,1)$, then $T[f(x)] = \sum_{k=0}^\infty a_k P^{Le}_k(x)$
    also has all of its zeros in the interval $(-1,1)$.
  \end{proposition}
  \noindent We now prove a conjecture of Chasse (Conjecture 222 of~\cite{Chasse-PhD-2011}).
  \begin{theorem}
    If $T$ is the linear transformation determined by $T[x^n]=P^{Le}_n(x)$, then
    \begin{align*}                                               Q^{Le}_{2k+1}(x)=0, \qquad
      Q_{2k}^{Le}(x) & =\frac{(2k-1)!!}{(2k)!!}(x^2-1)^k, \qquad k\geq 0.
    \end{align*}
    Hence,
    \begin{equation} \label{eqn:LegendereDO}
      T = \sum_{n=0}^{\infty} \frac{Q_n^{Le}(x)}{n!} D^n
      = \sum_{k=0}^{\infty} \frac{(2k-1)!!}{(2k)!!(2k)!} (x^2-1)^{k} D^{2k}.
    \end{equation}
  \end{theorem}
  \begin{proof}
    By successively applying the product rule, we obtain, for $0 \le k \le n$,
    \begin{align} \label{eqn:Dk}
      D^k[(x^2 - 1)^n] = \sum_{\ell = 0}^{\left \lfloor \frac{k}{2}\right \rfloor} (x^2 - 1)^{n - k +\ell} x^{k - 2\ell} \frac{n!}{(n - k +\ell)!}2^{k-\ell}C_{k,\ell},
    \end{align}
    where $C_{k,\ell}$ are constants.
    Indeed, by induction, we take a derivative of~\eqref{eqn:Dk} and obtain
    \begin{multline*}
      \sum_{\ell = 0}^{\left \lfloor \frac{k}{2}\right \rfloor}\Bigl[ (x^2 - 1)^{n - (k + 1) + \ell} x^{k + 1 - 2\ell} \frac{n!}{(n - (k + 1) - \ell)!} 2^{k + 1 - \ell}  \\*
      \qquad +(x^2 - 1)^{n - (k + 1) +\ell + 1}x^{k + 1 - 2(\ell + 1)} \frac{n!}{(n - (k + 1) + \ell + 1)!}2^{k + 1 - (\ell +1)}(k - 2\ell)\Bigr]C_{k,\ell},
    \end{multline*}
    which gives the following recursion for the constants $C_{k,\ell}$.
    \begin{equation} \label{eqn:C recursion}
      C_{k,\ell} =
      \begin{cases}
        0                                                  & k < 0 \text{ or } \ell < 0\text{ or }\ell > \left \lfloor \frac{k}{2}\right \rfloor \\
        1                                                  & k = \ell = 0                                                                        \\
        C_{k - 1,\ell} + (k - 2\ell + 1)C_{k - 1,\ell - 1} & \text{otherwise}.
      \end{cases}
    \end{equation}
    Setting $k = n$ in Equation~\eqref{eqn:Dk} and using the Rodrigues formula $P_n^{Le}(x) = \tfrac{1}{2^n n!}D^n[(x^2 - 1)^n]$ for the Legendre polynomials, we obtain
    \begin{align*}
      T[x^n] = P_n^{Le}(x) = \sum_{\ell = 0}^{\left \lfloor \frac{n}{2}\right \rfloor}\frac{C_{n,\ell}}{2^\ell \ell!}(x^2 - 1)^\ell x^{n - 2\ell} = \sum_{\ell = 0}^{\left \lfloor \frac{n}{2}\right \rfloor} \frac{(n - 2\ell)!C_{n,\ell}}{2^\ell \ell! n!}(x^2 - 1)^\ell D^{2\ell} x^n.
    \end{align*}
    Hence, to prove~\eqref{eqn:LegendereDO}, it suffices to show that
    \begin{equation} \label{CNl explicit}
      C_{n,\ell} = \frac{(2\ell - 1)!!\, n!\, \ell!\, 2^\ell}{(2\ell)!!\,(n - 2\ell)!(2\ell)!},
    \end{equation}
    so we show that the formula~\eqref{CNl explicit} satisfies the recursion given in~\eqref{eqn:C recursion}. Indeed,
    \begin{align*}
      & C_{n - 1,\ell} + (n - 2\ell + 1) C_{n - 1,\ell - 1}                                                                                                                                                              \\[1em]
      & = \frac{(2\ell - 1)!!\,(n - 1)!\,\ell!\,2^\ell}{(2\ell)!!\,(n - 1 - 2\ell)!\,(2\ell)!} + (n - 2\ell + 1)\frac{(2\ell - 3)!!\,(n - 1)!\,(\ell - 1)!\, 2^{\ell - 1}}{(2\ell - 2)!!(n  + 1 - 2\ell)!\,(2\ell - 2)!} \\[1em]
      & =\frac{(2\ell - 1)!!\, n!\, \ell!\, 2^\ell}{(2\ell)!!\,(n - 2\ell)!\,(2\ell)!}\left(\frac{n - 2\ell}{n} + \frac{2\ell}{n}\right) = C_{n,\ell}. \qedhere
    \end{align*}
  \end{proof}
  
  In summary, the linear transformations $T \colon \R[x]\to\R[x]$ determined by $T[x^n]=P_n(x)$ where $P_n(x)$ is the $n$\textsuperscript{th} physicists' Hermite, Laguerre, Chebyshev, or Legendre polynomial have coefficient polynomials whose roots are all real.  Furthermore, in the cases of Hermite and Laguerre transformations, the coefficient polynomials have interlacing roots.
  
  \section{Open Problems and Further Research} \label{section:open_prob}
  
  Finally, we state several open problems that naturally arise from the study of the topic of this paper.
  
  Proposition~\ref{Borcea-Branden-prop} due to Borcea and Br\"and\'en~\cite{Borcea-Branden-2009} classifies all linear operators that preserve real-rootedness. A natural problem following the results of the present paper is:
  \begin{problem} \label{problem-1}
    Classify all linear operators that preserve real-rootedness and are of the form $T[x^n]=P_n(x)$, where $\{P_n(x)\}$ is an OPS. In general, we do not expect an OPS to satisfy easily accessible formulas as is the case with a classical OPS.
  \end{problem}

  With Propositions~\ref{prop:Iserles-Saff} and~\ref{prop:Legendre-zeros} in mind, we present the following problem.
  \begin{problem}
    Does the measure for which an OPS $\{P_n(x)\}_{n=0}^{\infty}$ is an orthogonal sequence relate to the real-root preserving property of the transformation $x^n \mapsto P_n(x)$ in a meaningful way?
  \end{problem}

  Forg\'acs and Piotrowski \cite[Thm~4, p.469]{Forgacs-Piotrowski-Hermite-2015} showed that if $\{\gamma_n\}$ is a sequence such that $H_n(x) \mapsto \gamma_n H_n(x)$ is a real root preserving transformation, then the coefficient polynomials $Q_n(x)$ have real roots. Computer calculations have suggested that the roots of these $Q_n(x)$ interlace. We now present the following problem.
  \begin{problem}
    For what, if any, sequences $\{\gamma_n\}$ do the coefficient polynomials $Q_n(x)$ discussed above have interlacing roots?
  \end{problem}
  
  \begin{problem}
    Equation~\eqref{eqn:differential-difference-equation-for-Laguerre-Qk} in
    Lemma~\ref{lemma:differential-difference-equation-for-Laguerre-Qk}
    \[
    (n+1) Q^{L}_{n+1}(x) = (x-2x^2+x^3)(Q^{L}_n)'(x) + (n+1-2x -nx^2)Q^{L}_n(x)
    \]
    was found by guessing that
    \[
    (n+1)Q^{L}_{n+1}(x) = A_n(x) (Q^{L}_n)'(x) + B_n(x) Q^{L}_n(x),
    \]
    where $A_n(x)$ and $B_n(x)$ are polynomials, and then solving for $A_n(x)$ and $B_n(x)$ using a computer algebra system.  The proof of Lemma~\ref{lemma:differential-difference-equation-for-Laguerre-Qk} merely verifies that equation~\eqref{eqn:differential-difference-equation-for-Laguerre-Qk} is true. It would be nice to find a proof of Lemma~\ref{lemma:differential-difference-equation-for-Laguerre-Qk} that results more naturally from basic manipulations of properties of the Laguerre polynomials $L_n(x)$.
  \end{problem}
  
  \begin{problem}
    Determine what extra hypothesis is needed to ensure that if $T$ is a real-root preserver, then the $Q_n(x)$ have real interlacing roots?  (Note that, by Chasse \cite[Ex.~211]{Chasse-PhD-2011}, an operator of the form $T=\sum_{n=0}^{\infty}\frac{Q_n(x)}{n!}D^n$ where the $Q_n(x)$ have real interlacing roots is not necessarily a real-root preserving operator.)
  \end{problem}
  
  \section{Acknowledgments}
  We thank Petter Br\"and\'en, Matthew Chasse, Tam\'as Forgacs, and Andrzej Piotrowski for helpful discussions about the topics in this paper. The second author was partially supported by Timo Sepp\"al\"ainen through NSF Grant DMS-1854619.
  
  \bibliographystyle{amsplain}

\end{document}